\documentclass[a4paper]{article}

\usepackage[english]{babel}

\usepackage[utf8]{inputenc}
\usepackage[utf8]{inputenc}
\usepackage{amsmath}
\usepackage{graphicx}
\usepackage{amssymb}
\usepackage{amsthm}
\usepackage{tikz-cd}
\usepackage{mathrsfs}
\usepackage[colorinlistoftodos]{todonotes}
\usepackage{enumitem}
\usepackage{yfonts}
\usepackage{ dsfont }
\usepackage{MnSymbol}
\usepackage{slashed}

\title{Quantitative Thomas-Yau uniqueness}

\author{Yang Li}

\date{\today}
\newtheorem{thm}{Theorem}[section]
\newtheorem{lem}[thm]{Lemma}

\theoremstyle{definition}

\newtheorem{cor}[thm]{Corollary}

\newtheorem{rmk}{Remark}
\newtheorem{prop}[thm]{Proposition}
\newtheorem{Def}[thm]{Definition}
\newtheorem{Question}{Question}

\newtheorem*{Acknowledgement}{Acknowledgement}

\newcommand{\ie}{\emph{i.e.} }
\newcommand{\cf}{\emph{cf.} }

\newcommand{\R}{\mathbb{R}}
\newcommand{\C}{\mathbb{C}}
\newcommand{\Z}{\mathbb{Z}}

\newcommand{\Q}{\mathbb{Q}}

\newcommand{\norm}[1]{\left\lVert#1\right\rVert}
\newcommand{\Lap}{\Delta}

\begin{document}
	\maketitle

\begin{abstract}
	Under Floer theoretic conditions, 
we obtain quantitative estimates on the closeness (Hausdorff distance, flat norm, F-metric) between two Lagrangians, depending on the smallness of Lagrangian angles. Some applications include a strong-weak uniqueness theorem for special Lagrangians, and a characterization of varifold convergence to special Lagrangians in terms of Lagrangian angles.
\end{abstract}

\section{Background and introduction}

Inside an $n$-dimensional compact K\"ahler manifold $(X, \omega)$ with a nowhere vanishing holomorphic volume form $\Omega$ such that $\frac{\omega^n}{n!}=e^{2\rho}\frac{i^{n^2}}{2^n}\Omega\wedge \overline{\Omega}$, an $n$-dimensional compact oriented submanifold $L$ is called \textbf{special Lagrangian}, if 
\begin{equation}
\omega|_L=0, \quad \text{Im} \Omega|_L=0.
\end{equation}
These sit at the crossroad of minimal surface theory and symplectic geometry:
\begin{itemize}
\item If the metric is Calabi-Yau (namely $\rho=0$), then special Lagrangians are absolute volume minimizers inside their homology classes.

\item Lagrangian submanifolds equipped with (unobstructed) brane structures define objects inside the Fukaya category. 
\end{itemize}

\subsection{Thomas-Yau uniqueness}

As part of a wider program to relate the existence and uniqueness questions of special Lagrangian branes, to Fukaya categories and stability condition, Thomas and Yau \cite{ThomasYau} proved a remarkable uniqueness theorem, which from the modern perspective reads

\begin{thm}\cite{ThomasYau}\cite{JoyceImagi}\cite{Imagi}\label{ThomasYauuniqueness}
Let $L,L'$ be two compact embedded special Lagrangians with unobstructed brane structures, defining isomorphic objects in the derived Fukaya category, then their supports coincide.
\end{thm}

Recall the \textbf{Lagrangian angle} $\theta: L\to \R$ is defined by $\Omega|_L= e^{-\rho} e^{i\theta}dvol_L$, so special Lagrangians amount to the condition $\theta=0$.  The \textbf{Floer degree} of a transverse intersection point $p\in CF^*(L_1,L_2)$ between two Lagrangians $L_1, L_2$ is
\begin{equation}\label{degree}
\mu_{L_1,L_2} (p)=\frac{1}{\pi} (\sum_1^n \alpha_i - \theta_{L_2} (p) + \theta_{L_1} (p)  )\in \Z,
\end{equation}
where inside $T_pX\simeq \C^n$ we can put the tangent spaces into the standard form
\[
T_p L_1= \R^n, \quad T_pL_2= (e^{i\alpha_1},\ldots e^{i\alpha_n})\R^n,\quad 0<\alpha_i<\pi.
\]

The proof idea of Theorem \ref{ThomasYauuniqueness} can be outlined as follows. Assume $L\neq L'$.

\begin{itemize}
\item By making $C^\infty$-small Hamiltonian perturbations of $L,L'$, we can replace $L,L'$ by isomorphic objects $L_1,L_2$ \emph{intersecting transversely}. Morever, either by a Morse theoretic argument \cite{ThomasYau}, or using real analyticity via a Lojasiewicz inequality \cite{JoyceImagi}, 
one can remove the degree zero intersection points, to ensure the Floer cochain group $CF^0(L_1,L_2)=0$.

\item Consequently the Floer cohomology $HF^0(L,L)\simeq HF^0(L_1,L_2)=0$. In particular the unit of the Floer cohomology ring vanishes, which is impossible, because it violates Poincar\'e duality (alternatively, because the image of the unit under the open-closed map is the homology class $[L]\in H_n(X, \Lambda_{nov})$, which cannot be zero).
\end{itemize}

We take note of a few conceptual features:
\begin{itemize}
\item The proof of Thomas-Yau uniqueness is analogous to a \emph{strong maximum principle} argument. The role of symplectic topology is to force the existence of a degree zero intersection point, which is analogous to the existence of a maximum, and a local calculation concerning the intersection point results in a contradiction.

\item The metric is not necessarily Calabi-Yau. The compactness of $X$ can be replaced by any other standard settings where Floer theory makes sense.

\item  The proof is not sensitive to the details of the Fukaya category. 

\item The known proofs rely essentially on smoothness assumptions of the Lagrangians.

\end{itemize}

\subsection{Quantitative Thomas-Yau uniqueness}

A common theme in geometric analysis is \emph{rigidity theorems}, e.g. a strong maximum principle naturally suggests a Harnack inequality. Analogously, we ask

\begin{Question}
Suppose $L,L'$ have \emph{small} Lagrangian angles (eg. $\norm{\theta}_{C^0}\ll 1$, or $\norm{\theta}_{L^1}\ll 1$), then do they have to be \emph{uniformly} close to each other (eg. in the Hausdorff distance for subsets, or in the flat norm distance for integral currents, or the $F$-metric for varifolds)?
\end{Question}

\begin{rmk}
The very recent paper of Abouzaid and Imagi \cite{AbouzaidImagi} studies symplectic topological consequences of special Lagrangians lying inside a small \emph{$C^0$-neighbourhood} of a given special Lagrangian, under additional hypotheses on the fundamental group and its representations.  
\end{rmk}

We now indicate why new ideas are needed for this question. In a na\"ive strategy, one takes a sequence of $L_i, L_i'$, with Lagrangian angles converging to zero, and attempts to use compactness theorems in geometric measure theory to extract subsequential limits $L_\infty,L'_\infty$, which should be special Lagrangian integral currents. To deduce that $L_i$ is close to $L_i'$, one would like to prove $L_\infty=L'_\infty$. This would require a \emph{singular Lagrangian} version of the Thomas-Yau uniqueness theorem, which is yet unknown. (The interested reader may see \cite[section 5.7]{Li} for some heuristic ideas).

Our main result is
\begin{thm}\label{main}
Let $L,L'$ be two compact, smoothly embedded Lagrangians with unobstructed brane structures, in the same fixed homology classes in $H_n(X,\Q)$, such that $HF^0(L,L')\neq 0$ or $HF^0(L',L)\neq 0$ holds. Assume there is some small enough $\epsilon$, such that the Lagrangian angles satisfy
\begin{itemize}
\item $\norm{\theta_L}_{C^0} \leq \epsilon\ll 1.$
\item $\norm{\theta_{L'}}_{C^0}\leq \frac{\pi}{2}-\epsilon_0$ for some fixed $\epsilon_0>0$, and $\text{Vol}(\{ \theta_{L'}>\epsilon \})\ll \epsilon^n$.
\end{itemize}
Then the Hausdorff distance between $L, L'$ is uniformly small:
\begin{equation}
\begin{cases}
\sup_{p\in L} \text{dist}(p,L') \leq C\epsilon^{\frac{1}{2n}},
\\
\sup_{p\in L'} \text{dist}(p,L)\leq C\epsilon^{ \frac{1}{4n^2}  }.
\end{cases}
\end{equation}
The flat norm distance between $L$ and $L'$ is bounded uniformly by $C\epsilon^{ \frac{1}{4n}+ \frac{1}{4n^2}  }$. The $F$-metric between $L,L'$ is bounded by $C\epsilon^{ \frac{1}{8n}  }$. 
The constants depend on $\epsilon_0$, $X$, and the $C^2$-regularity bounds on $L$, but not on the regularity bounds on $L'$, nor on the small $\epsilon$. 
\end{thm}

A few explanations may clarify the significance of the assumptions:

\begin{itemize}
\item If $L,L'$ are isomorphic nonzero objects in the derived Fukaya category, then $HF^0(L,L')\neq 0$, $HF^0(L',L)\neq 0$ and $[L]=[L']\in H_n(X, \Q)$.

\item The condition $\norm{\theta_{L'}}_{C^0}\leq \frac{\pi}{2}-\epsilon_0$ is called \textbf{quantitative almost calibratedness}, important in the Thomas-Yau program (\cf \cite{ThomasYau}\cite{Li}). Obviously, this is weaker than assuming the $C^0$-norm on $\theta_{L'}$ to be small. One of its immediate consequences is the a priori homological mass bound
\[
\text{Vol}(L')=\int_{L'} dvol \leq \frac{1}{\sin \epsilon_0} \int_{L'} \text{Re}\Omega .
\]

\item The weak $L^1$-type bound $\text{Vol}(\{ \theta_{L'}>\epsilon \})\ll \epsilon^n$
 is implied by $\norm{\theta_{L'} }_{L^1(L')} \ll \epsilon^{n+1} $. This is again weaker than the $C^0$-smallness of $\theta_{L'}$. This weakening is attractive because if a sequence $L_i'$ of quantitatively almost calibrated Lagrangians converge to a special Lagrangian in the weak topology of varifolds (\ie as Radon measures on the Grassmannian bundle), then $\norm{\theta_{L'_i} }_{L^1}\to 0$ as $i\to \infty$, but $\norm{\theta_{L'_i}}_{C^0}$ does not necessarily converge to zero.
\end{itemize}

Theorem \ref{main} can be interpreted as the quantitative version of a 
\textbf{strong-weak uniqueness} theorem. 

\begin{Def}
Given a sequence of smooth embedded unobstructed Lagrangian branes $L_i$ which lie in a fixed nonzero derived Fukaya category class (so in particular the same rational homology class), and are all quantitatively almost calibrated. We say a Lagrangian $C^2$-submanifold (resp. Lagrangian integral current) $L_\infty$ is a \textbf{$C^2$-limit} (resp. \textbf{varifold limit}) if $L_i$ converge to $L_\infty$ in the $C^2$-topology (resp. in both the weak topology of varifolds and the flat norm topology of integral currents). Notice there is no requirement that $L_\infty$ is equipped with any brane structure. 
\end{Def}

\begin{cor}(Strong-weak uniqueness)\label{strongweak} Fix a derived Fukaya category class, and let $L_\infty$ be a $C^2$-limit and $L_\infty'$ a varifold limit. Assume both are special Lagrangian $\theta_{L_\infty}=\theta_{L'_\infty}=0$, then $L_\infty=L_\infty'$ as integral currents.
\end{cor}

\begin{proof}
From the $C^2$-limit assumption, $L_i\to L_\infty$ with $\norm{\theta}_{C^0}\to 0$. From the varifold limit assumption, $L_i'\to L$ with $\norm{\theta}_{L^1}\to 0$. All Lagrangians lie in the same homology class, and all the approximants $L_i, L_i'$ lie in the same derived category class. Thus Theorem \ref{main} applies to $L_i, L_i'$ for $i\gg 1$ and arbitrarily small $\epsilon$, and the uniform $C^2$-regularity on $L_i$ means that all the constants are uniform. The flat norm distance between $L_i, L_i'$ is bounded by $C\epsilon^{ \frac{1}{4n}+ \frac{1}{4n^2}  }$. Taking the limit as $i\to +\infty$, the flat norm distance between $L_\infty$ and $L_\infty'$ is bounded by $C\epsilon^{ \frac{1}{4n}+ \frac{1}{4n^2}  }$, and letting $\epsilon\to 0$ gives $L_\infty=L_\infty'$. 
\end{proof}

Provided a derived Fukaya category class admits a special Lagrangian $C^2$-limit, we have a satisfactory characterization of varifold topology convergence in terms of the Lagrangian angle function.

\begin{cor}
Fix a derived Fukaya category class, and assume there is a sequence of smooth embedded representatives $L_i$, converging in $C^2$ topology to a special Lagrangian $L_\infty$. Let $L_i'$ be a sequence of quantitatively almost calibrated Lagrangian branes in the same derived Fukaya category class. Then $L_i'$ converges in the varifold topology to $L_\infty$ if and only if $\norm{\theta_{L_i'}}_{L^1}\to 0$.

\end{cor}

\begin{proof}
Since the Lagrangian angle is a continuous function on the open subset $\{ \text{Re}\Omega>0  \}$ in the Grassmannian bundle, the weak topology convergence of $L_i$ implies 
\[
\int_{L_i} |\theta| dvol_{L_i}\to \int_{L_\infty} |\theta| dvol =0. 
\]
Conversely, if $\norm{\theta_{L_i'}}_{L^1}\to 0$, then Theorem \ref{main} applies to $L_i$ and $L_i'$ for $i\gg 1$ and arbitrarily small $\epsilon$, whence the $F$-metric between $L_i$ and $L_i'$ is bounded by $C\epsilon^{1/8n}$ for $i\gg 1$ depending on $\epsilon$. In the $i\to +\infty$ limit, the $F$-metric distance between $L_i'$ and $L_\infty$ tends to zero. Since all Lagrangians have uniform mass bounds, the $F$-metric induces the same topology as the weak topology on varifolds, whence $L_i'\to L_\infty$ in the varifold topology.
\end{proof}

\begin{rmk}
Since in our setting $\norm{\theta}_{L^\infty}$ and the mass of the Lagrangian both have uniform bounds, by interpolation $\norm{\theta}_{L^1}\to 0$ is equivalent to $\norm{\theta}_{L^p}\to 0$ for any fixed $p>0$.
\end{rmk}

\subsection{Open questions}

We raise some natural questions:

\begin{Question}
With the same assumptions on $L, L'$ as in Thm \ref{main},  what would be the \emph{sharp exponent} for the Hausdorff distance bound in terms of $\epsilon$? 
\end{Question}

This question has many variants. For instance, one can replace $\text{Vol}(\{ \theta_{L'}>\epsilon  \})\leq \epsilon^n$ with the smallness of  $\norm{\theta_{L'}}_{L^p(L')}$ (resp. $\norm{\theta_{L'}}_{C^0(L')}$). We expect the exponents of $\epsilon$ in Theorem \ref{main} to be not optimal. On the other hand, even in the simple setting of Calabi-Yau metrics, assuming $L$ is a fixed smooth special Lagrangian, and $\norm{\theta_{L'}}_{L^1}<\epsilon$, we do not expect that the Hausdorff distance to be uniformly bounded by $O(\epsilon)$. The reason is that the linearization of the special Lagrangian equation is the Poisson equation $\Lap u= f$, and the failure of the Sobolev embedding $W^{2,1}\to C^1$ means we cannot expect $u$ to have $C^1$-bounds proportional to $\norm{f}_{L^1}$. Some nonlinear effects are necessary for a result like Thm. \ref{main}.

\begin{Question}
Do the uniform constants really depend on the $C^2$-regularity bounds of $L$? 
\end{Question}

The part of our arguments most sensitive to $C^2$-regularity bounds is Lemma \ref{smallC0density}, which involves solving the Poisson equation on $L$ with bounds. If $L$ has some mild degeneration, eg. when it is approximately the desingularization of some special Lagrangian with certain isolated conical singularities, then we expect results similar to Theorem \ref{main} would hold, with possibly different exponents for $\epsilon$. On the other hand, it is not clear if we can drop all regularity bounds on $L$ beyond the information of the homology class and the Lagrangian angles. A closely related question is whether \emph{strong uniqueness} holds:

\begin{Question}
Consider the variant of Cor. \ref{strongweak} where both $L_\infty$ and $L_\infty'$ are varifold limits. Assume both are special Lagrangian $\theta_{L_\infty}=\theta_{L'_\infty}=0$, then do they coincide as integral currents?
\end{Question}

As we will indicate in a series of remarks, the embedded assumption on the Lagrangians in Theorem \ref{main} can be relaxed to immersed Lagrangians with \emph{connected domains} and transverse self intersection points, with rather minor changes. The connected domain assumption is quite crucial to our arguments here (\cf Remark \ref{connecteddomain}), and also appears in the literature on Thomas-Yau uniqueness for immersed Lagrangians \cite{Imagi}. On the other hand, the hypothesis of Theorem \ref{main} is a little weaker than assuming $L,L'$ to lie in the same derived Fukaya category class, and does not make explicit use of holomorphic curves.

\begin{Question}
If $L,L'$ are immersed Lagrangians defining the same object in the derived Fukaya category, not necessarily with connected domains, then is there still an analogue of the quantitative Thomas-Yau uniqueness theorem? Is there a corresponding strong-weak uniqueness theorem?	
\end{Question}

\section{Proof of quantitative Thomas-Yau uniqueness}

Throughout this section,
let $L, L'$ be two compact, smoothly embedded Lagrangians with unobstructed brane structures, such that $HF^0(L, L')\neq 0$ or $HF^0(L',L)\neq 0$.  All constants $C$ depend only on $X$ and the $C^2$-regularity bounds of $L$, but not on the regularity bounds for $L'$.

\subsection{Small $\norm{\theta}_{C^0}$ case}

In the special case where $\norm{\theta}_{C^0}\ll 1$ is small, we show that the set of points on $L$ lying near $L'$ is sufficiently dense. This argument contains the Floer theoretic ingredient in our strategy. We do not yet impose $[L]=[L']\in H_n(X)$.

\begin{lem}\label{smallC0density}
Assume $\norm{\theta}_{C^0}\leq \epsilon\ll 1$ for both $L, L'$. 
For all $p\in L$, there is some $p'\in L\cap B(C_1\epsilon^{1/2n})$, such that $\text{dist}(p', L')\leq C\epsilon^{\frac{n+1}{2n}  }	 $. In particular $\text{dist}(p,L')\leq C\epsilon^{ 1/2n  }$.
\end{lem}

\begin{proof}
We perform a Hamiltonian perturbation of the $L$ within the Weinstein tubular neighbourhood $T^*L$ of the Lagrangian $L$.
For a smooth Hamiltonian function $h: X\to \R$, the time $\epsilon$ flow sends $L$ to $\varphi_\epsilon(L)$ whose Lagrangian angle $\theta_{\varphi_\epsilon(L)}$ satisfies
\[
|\theta_{\varphi_\epsilon(L)}- \theta_L -  \epsilon \mathcal{L}(h) |\leq  C\epsilon^2(\norm{dh}_{C^1}^2+  \norm{dh}_{C^1}),
\]
where the linear operator $\mathcal{L}$ is the weighted Laplacian prescribed by
\[
\mathcal{L} (h) e^{-\rho}dvol =  d(e^{-\rho}*_L dh),
\]
which in the Calabi-Yau metric case ($\rho=0$) reduces to the Laplacian.

Given any point $p\in L$, we prescribe a smooth function $f:L\to \R$, such that $f= 4$ outside  $B(p, \eta)$ for some $\eta\ll 1$ to be determined, the weighted integral $\int_L e^{-\rho} fdvol_L=0$, and $|f|+ | df|\eta \leq C \eta^{-n}$ inside $B(p,\eta)$. We solve $\mathcal{L}(h)=f$ with $\int_L he^{-\rho}dvol_L=0$, to find the smooth function $h$ with bound
$
|dh|+ \eta|\nabla dh| \leq C \eta^{1-n}.
$
Using the smoothness of $L$, we can extend $h$ to a smooth function on $X$ with the same bounds.

We choose $\eta=C_1 \epsilon^{1/2n}$ for some large constant $C_1$ independent of $\epsilon$. Then the dispacement of the time $\epsilon$-flow is bounded by $C\epsilon|dh|\leq C\epsilon \eta^{1-n}\leq C \epsilon^{ \frac{n+1}{2n} } \ll 1$, so $\varphi_\epsilon(L)$ is well defined inside the Weinstein neighbourhood. Furthermore, 
\[
|\theta_{\varphi_\epsilon(L)}- \theta_L -  \epsilon\mathcal{L}(h) |\leq C\epsilon^2(\norm{dh}_{C^1}^2+  \norm{dh}_{C^1})\leq  C( \epsilon \eta^{-n}  )^2 \ll \epsilon,
\]
whence on $\varphi_\epsilon( L\setminus B(p, \eta)  )$, 
\[
\theta_{\varphi_\epsilon(L)} \geq \theta_L+ \epsilon\mathcal{L}(h)- \epsilon \geq \epsilon\mathcal{L}(h)- 2\epsilon \geq 4\epsilon- 2\epsilon \geq 2\epsilon. 
\]
Without loss $\varphi_\epsilon(L)$ is transverse to $L'$ by genericity.
By formula (\ref{degree}),
any Lagrangian intersection $q\in CF^*(\varphi_\epsilon(L),L')$ has Floer degree 
\[
\mu(q)\geq \frac{1}{\pi}( \theta_{\varphi_\epsilon(L)}- \theta_{L'} )\geq \frac{1}{\pi}\epsilon,
\]
 hence degree zero intersections cannot occur on $\varphi_\epsilon( L\setminus B(p, \eta)  )$.
By the Hamiltonian invariance of Floer cohomology \[
HF^0(\varphi_\epsilon(L), L')=HF^0(L, L')\neq 0.
\]
This forces there to be an intersection $q\in \varphi_\epsilon( L\cap B(p, \eta)  )\cap L'$. Thus there is $p'\in L\cap B(p,\eta)$ whose distance to $q$ is less than $C\epsilon|dh|\leq C\epsilon^{ \frac{n+1}{2n}  }$.
\end{proof}

\begin{rmk}\label{connecteddomain}
If we replace the embedded Lagrangians by immersed Lagrangians with connected domain and transverse self intersections, we can still use the Floer cohomology of Akaho-Joyce \cite{AkahoJoyce}. The main caveat is that in order for the Hamiltonian function $h: L\to \R$ to extend to $X$, we need $h$ to take the same value at the finitely many self intersection points. These finitely many linear constraints are easy to meet, by relaxing $f=4$ outside $B(p,\eta)$ to the more flexible condition $f\geq 4$  outside $B(p,\eta)$.

The connected domain hypothesis on $L$ is important for solving the equation $\mathcal{L}(h)=f$. If we drop this hypothesis, then take for instance $L$ the disjoint union of two special Lagrangians, and $L'$ to be one of the components, and the lemma would be false.
\end{rmk}

\begin{rmk}
Lemma \ref{smallC0density} easily gives a \emph{new proof} to the original Thomas-Yau uniqueness theorem \ref{ThomasYauuniqueness}, not relying on Morse theory or real analyticity. Indeed, by taking the $\epsilon\to 0$ limit, we deduce that each point on $L$ has zero distance to $L'$, hence $L\subset L'$. Since the roles of $L,L'$ are symmetric in theorem \ref{ThomasYauuniqueness}, we recover $L=L'$. The rest of this paper treats the additional difficulties caused by giving up quantitative regularity bounds on $L'$, and by relaxing the $C^0$-smallness of $\theta_{L'}$ to merely smallness on most of the measure.
\end{rmk}

\subsection{Generic perturbation technique}

The technique of controlling the number of intersection points by utilizing a sufficiently rich family of perturbations, originated from Arnold \cite{Arnold}, who used it to study the dynamical growth of intersection points.  Prop. \ref{Arnold} is a detailed exposition to clarify the dependence of constants. Seidel \cite[Lecture 5]{Seidel} contains an application to bound the rank of Lagrangian Floer cohomology.

Let $M$ be an $m$-dimensional compact manifold, and $N, N'$ be two compact submanifolds of complementary dimension, and let $U\subset N'$ be an open subset. 
Let $V$ be a $p$-dimensional space of vector fields on $M$, which is surjective to the normal space $TM/TN$ at every point of $N$. Fix a Euclidean metric on $V$, and consider the $p$-dimensional family of diffeomorphism $\phi_t$ of $M$ obtained by exponentiating the vector fields in $P=B(0,\epsilon)\subset V$ for some small  $\epsilon$.  This induces a $p$-parameter deformation family 
$\mathcal{N}\to P$ for $N$, whose fibres are $\phi_t(N)$ for $t\in P$. There is an evaluation map $\pi: \mathcal{N}\to M$, which by construction is surjective on tangent spaces.

\begin{prop}\label{Arnold}
There exists some $t\in P$, such that $\phi_t(N)$ intersects $N'$ transversely, and the number of intersection points 
\[
|\phi_t(N)\cap U| \leq C\epsilon^{-n} \text{Vol}(U), 
\]	
where the constant $C$ depends only on $M$, $V$, and the $C^1$-regularity of $N$, but not on the small $\epsilon$ and the regularity bounds of $N'$.
\end{prop}

\begin{proof}
The transversality holds for a.e. $t\in P$ by Sard-Smale, so that $|\phi_t(N)\cap U|$ is a well defined integer for a.e. $t\in P$.
The integral 
\[
\int_P |\phi_t(N)\cap U| dt \leq \text{Vol}(\mathcal{N}\cap \pi^{-1}(U)).
\]
Now $\pi: \mathcal{N}\cap \pi^{-1}(U)\to U$ is surjective on tangent spaces, with Jacobian factor bounded below by $C^{-1}$, so 
\[
\text{Vol}(\mathcal{N}\cap \pi^{-1}(U))\leq C\int_U  \text{Vol}( \mathcal{N}\cap \pi^{-1}(y) ) dy.
\]
For each $y\in U$, the contributions to  $\mathcal{N}\cap \pi^{-1}(y) $ come from the deformations of the small local region $N\cap B(y, C\epsilon)$. The assumption that $V$ is surjective to $TM/TN$ at every point of $N$, together with the $C^1$-regularity of $N$, allows us to find a codimension $n$ subspace $V_y\subset V$, such that the projection
\[
\mathcal{N}\cap \pi^{-1}(y) \to P=B(0,\epsilon)\subset V\to V_y
\]
exhibits $\mathcal{N}\cap \pi^{-1}(y)$ as part of a $C^1$-graph over the $\epsilon$-ball inside $V_y$, whose volume is bounded by $C\epsilon^{p-n}$. We deduce 
\[
\text{Vol}(\mathcal{N}\cap \pi^{-1}(y) ) \leq C \epsilon^{p-n}.
\]
Combining the above,
\[
\int_P |\phi_t(N)\cap U| dt \leq C\epsilon^{p-n} \text{Vol}(U),
\]
so we can select some $t\in P$ with
\[
|\phi_t(N)\cap U| \leq C \text{Vol}(U)\epsilon^{p-n} (\text{Vol}(P))^{-1} \leq C\text{Vol}(U)\epsilon^{-n}.
\]
\end{proof}

In our application, the idea is that under sufficiently generic Hamiltonian deformation, subsets with very small measure do not contribute to Lagrangian intersection points. 
This allows us to relax Lemma \ref{smallC0density} to a weak $L^1$-assumption on the Lagrangian angle of $L'$.

\begin{prop}
Assume $\norm{\theta_L}_{C^0}\leq \epsilon$, while $\text{Vol}( \{  |\theta_{L'} > \epsilon      \}   )\ll \epsilon^{-n}$. 
For all $p\in L$, there is some $p'\in L\cap B(C_1\epsilon^{1/2n})$, such that $\text{dist}(p', L')\leq C\epsilon^{\frac{n+1}{2n}  }	 $.
\end{prop}

\begin{proof}
We can find a large dimensional vector space $V$ of Hamiltonian vector fields, which is surjective to the normal space $TX/TL$ at every point of $L$. We can pick the Euclidean metric on $V$, so that for any $v$ in the unit ball of $V$, the induced vector field acting on the Grassmannian bundle over $X$ has $C^0$-norm $\ll 1$. In particular, the Hamiltonian diffeomorphisms $\phi_t$ obtained via exponentiating vector fields in $ P=B(0,\epsilon)\subset V$, only change the Lagrangian angle function by an amount $\ll \epsilon$.

We now re-examine the argument of Lemma \ref{smallC0density}. For any $p\in L$, we can construct the $C^1$-small Hamiltonian deformation $\varphi_\epsilon(L)$ of $L$, such that on $\varphi_\epsilon(L)(L\setminus B(p,\eta))$, we have $\theta_{\varphi_\epsilon(L)}\geq 2\epsilon$. Thus for any $t\in P$, the Hamiltonian deformation $\phi_t\varphi_\epsilon(L)$ satisfies away from $B(p,\eta)$
\[
\theta_{\phi_t\varphi_\epsilon(L)} \geq \theta_{\varphi_\epsilon(L)}- \frac{\epsilon}{2} \geq \frac{3}{2}\epsilon.
\] 
Let $U=\{  \theta_{L'}> \epsilon   \}\subset L'$, then according to Prop. \ref{Arnold}, there is some $t\in P$, such that $\phi_t\varphi_\epsilon(L)$ is transverse to $L'$, and the number of intersection points
\[
|\phi_t\varphi_\epsilon(L)\cap U| \leq C\epsilon^{-n} \text{Vol}(  U ) \ll 1. 
\]
Thus $\phi_t\varphi_\epsilon(L)\cap U$ is in fact \emph{empty}. This forces $\theta_{L'}\leq \epsilon$ at the Lagrangian intersections $\phi_t\varphi_\epsilon(L)\cap L'$, and we conclude as in Lemma \ref{smallC0density}.
\end{proof}

%The following consequence is worth noting:

%\begin{cor}
%Let $L_1, L_2$ be unobstructed graded Lagrangian branes inside a fixed symplectic Calabi-Yau manifold, defining objects in the Fukaya category. Then the total rank of the Floer cohomology
%\[
%\text{rk} HF^*(L_1,L_2) \leq C \text{Vol}(L_1)\text{Vol}(L_2),
%\]
%where the constant $C$ only depends on $X$ but not on $L_1, L_2$.
%In particular, the rank of the self Floer cohomology $\text{rk} HF^*(L,L) \leq C \text{Vol}(L)^2.$
%\end{cor}

%\begin{proof}
%(courtesy of P. Seidel)

%...........

%\end{proof}

%\begin{rmk}
%The moral is that volume bounds the Floer theoretic complexity of Lagrangian branes, without any assumption on the mean curvature. This may be compared to the recent work of Antoine Song \cite{Song} bounding the total Betti numbers of area minimizing hypersurfaces in terms of the index.
%\end{rmk}

Consequently, we get one half of the Hausdorff distance estimate:

\begin{cor}\label{HalfHausdorff}
Under the same conditions,
\begin{equation}
\sup_{p\in L} \text{dist}(p,L') \leq C\epsilon^{1/2n}.
\end{equation}
\end{cor}

\begin{rmk}\label{C2regularity}
We have so far not used the condition $[L]=[L']\in H_n(X,\Q)$. Without this condition, the constant $C$ in Cor. \ref{HalfHausdorff} really requires some regularity bound on $L$. For instance, we consider a surgery exact triangle $L_1\to L_2\to L_3\to L_1[1]$, where $L_2$ is the Lagrangian connected sum of $L_1, L_3$ with very small neck, and all Lagrangians angles can be made arbitrarily small in $C^0$. The mere assumption that $HF^0(L_1,L_2)\neq 0$ cannot imply that $\sup_{p\in L_2} \text{dist}(p,L_1)$ is small. The proof breaks down because the $C^1$-regularity bound of $L_2$ is highly degenerate in the neck region.
\end{rmk}

%....................

%A variant of the above argument shows that the subset of $L'$ close to $L$ has nontrivial measure. 

\subsection{Monotonicity inequality}

In minimal surface theory, the famous monotonicity formula says that for an $n$-dimensional submanifold $N$ inside a smooth compact ambient manifold, if the mean curvature $\norm{\vec{H}}_{L^\infty}<+\infty$, then inside coordinate balls the volume ratio
\[
e^{Cr} r^{-n}\text{Vol}(B(p,r)) 
\]
increases with the radius $r$. One useful consequence is that the volume of $N\cap B(p,r)$ has a lower bound for $p\in N$. In this paper we \emph{only impose assumptions on the Lagrangian angle $\theta$, but not on $|\vec{H}|$ directly}, but the \textbf{volume lower bound} is still satisfied.  Our style of arguments are inspired by Neves  \cite[section 3.3]{Neves} (\cf also \cite[section 5.2]{Li}).

Recall a special case of the optimal isoperimetric inequality of Almgren. Denote $\omega_n$ as the volume of the $n$-dim unit ball in $\R^n$.

\begin{prop}\cite[Thm 10]{Almgrenisoperimetric}
Let $T$ be an $(n-1)$-dimensional integral current inside $\R^N$ with $\partial T=0$. Then there is an integral $n$-current $Q$ inside $\R^N$ with $\partial Q=T$ and 
\[
Mass(Q) \leq n^{ -\frac{n}{n-1}  } \omega_n^{-\frac{1}{n-1}} Mass(T)^{\frac{n}{n-1}}.
\]
\end{prop}

\begin{rmk}
The inequality is saturated by the $n$-dimensional unit ball; this sharpness of constant will be important. Morever, if $T$ is contained in some ball $B(R)$, then $Q$ can also be taken inside $B(R)$, because there is a retraction of $\R^N$ to $B(R)$ with Lipschitz constant one. 	
\end{rmk}

We now work inside the K\"ahler manifold $X$. Around any fixed $p\in X$, we can take local holomorphic coordinates $z_1,\ldots z_n$, such that for small $|z|$,
\[
\omega= \frac{\sqrt{-1}}{2}\sum dz_i\wedge d\bar{z}_i +O(|z|), \quad e^{\rho(p)}\Omega= dz_1\wedge \ldots dz_n +O(|z|), \quad g= \sum |dz_i|^2 +O(|z|).
\]
Beware that the smooth function $\rho$ measures the failure of the metric to be Calabi-Yau.

\begin{prop}\label{monotonicity}
	Let $L'$ be a smooth Lagrangian with $|\theta|\leq \frac{\pi}{2}-\epsilon_0$, and 
	there exists $p'\in L'$ with $|p'|\leq \lambda\ll 1$. We shall use $C(\epsilon_0)$ to denote constants depending only on $\epsilon_0$ and the coordinate chart.
\begin{enumerate}
	\item 
For the Euclidean balls $B_{Eucl}(r)$  with radius $r>\lambda$ contained in the chart, we have the volume lower bound $\text{Vol}(L'\cap B_{Eucl}(r))\geq C(\epsilon_0)(r-\lambda)^n$.
	
	\item Assume furthermore that $\text{Vol}( \{ \theta_{L'}> \epsilon   \}\cap L' ) \leq \lambda^n $ where $\lambda\geq \epsilon^2$. Then for  the Euclidean balls $B_{Eucl}(r)$  with radius $r>C(\epsilon_0)\lambda$ contained in the chart, we have the sharper volume lower bound
\begin{equation}
\text{Vol}(L'\cap B_{Eucl}(r)) \geq \omega_n (r-C(\epsilon_0)\lambda)^n (1-C(\epsilon_0)r).
\end{equation}
\end{enumerate}

\end{prop}

\begin{proof}
For a.e. $0<r$ smaller than the radius of the coordinate ball (of order $O(1)$), the level set $ L'\cap \partial B_{Eucl}(r)$ is smooth, and the isoperimetric inequality allows us to find $Q\subset B_{Eucl}(r)$ with $\partial Q=L'\cap \partial B_{Eucl}(r)$ and mass bound
\[
Mass(Q) \leq 
n^{ -\frac{n}{n-1}  } \omega_n^{\frac{-1}{n-1}}
 \mathcal{H}^{n-1}(L'\cap B_{Eucl}(r))^{\frac{n}{n-1}}.
\]
Since $\partial (L'\cap B_{Eucl}(r))= \partial Q$, the form $\text{Re}\Omega$ is closed, and $\text{Re}\Omega |_Q \leq (1+O(r))e^{-\rho(p)} dvol_Q$,
\[
\int_{ L'\cap B_{Eucl}(r)} \text{Re} \Omega = \int_{ Q} \text{Re}\Omega \leq (1+O(r))e^{-\rho(p)} Mass(Q).
\]
Under the quantitative almost calibrated condition $|\theta|\leq \frac{\pi}{2}-\epsilon_0$,
\[
\text{Vol}(L'\cap  B_{Eucl}(r)  ) = \int_{ L'\cap B_{Eucl}(r)} dvol \leq \frac{1}{\sin \epsilon_0} \int_{ L'\cap B_{Eucl}(r)}  e^{\rho}\text{Re} \Omega .
\]
Combining the above,
\[
\text{Vol}(L'\cap  B_{Eucl}(r)  )\leq C(\epsilon_0) \mathcal{H}^{n-1}(L'\cap B_{Eucl}(r))^{\frac{n}{n-1}}.
\]
The function $f(r)= \text{Vol}(L'\cap  B_{Eucl}(r)  )$ is increasing in $r$. 
Using the coarea formula, we rewrite the differential inequality as
\[
f(r)\leq C(\epsilon_0) f'(r)^{  \frac{n}{n-1} }.
\]
For $r>\lambda$, we have $f(r)>0$, and the increasing function $f$ satisfies
\[
(f^{1/n})' \geq C(\epsilon_0),
\]
whence $f\geq C(\epsilon_0) (r-\lambda)^n$.

Next we assume $\text{Vol}( \{ \theta_{L'}> \epsilon   \}\cap L' ) \leq \lambda^n $ and deduce the sharper volume lower bound. For $O(1)>r> C\lambda$ with $C$ depending only on the volume lower bound constant, we have  $f(r) >\lambda^n  $.
Notice that on $L'\cap \{ \theta_{L'}\leq  \epsilon   \}\cap B_{Eucl}(r)$,
\[
\int_{ L'\cap B_{Eucl}(r)\cap \{ \theta_{L'}\leq  \epsilon   \}} dvol \leq \frac{1}{\cos \epsilon} \int_{ L'\cap B_{Eucl}(r)\cap \{ \theta_{L'}\leq  \epsilon   \}} e^\rho \text{Re} \Omega \leq \frac{1}{\cos \epsilon} \int_{ L'\cap B_{Eucl}(r)} e^\rho \text{Re} \Omega  .
\]
Whence 
\[
\begin{split}
& f(r)\leq \lambda^n+ \frac{1}{\cos \epsilon} \int_{ L'\cap B_{Eucl}(r)} e^\rho \text{Re} \Omega \\
&
\leq 
\lambda^n+ \frac{1}{\cos \epsilon} (1+O(r)) \text{Mass}(Q)\\
&\leq 
\lambda^n+  \frac{1}{\cos \epsilon} (1+O(r))  n^{ -\frac{n}{n-1}  } \omega_n^{\frac{-1}{n-1}}
f'(r)^{ \frac{n}{n-1} }  .
\end{split}
\]
Now $\frac{1}{\cos \epsilon} = 1+O(\epsilon^2) =1+O(\lambda)=1+O(r)$, so the  $ \frac{1}{\cos \epsilon}$ factor can be absorbed into the $1+O(r)$ factor by changing the constant. Rearranging the terms, 
\[
f'(r) \geq n \omega_n^{1/n} (1- O(r)) (f-\lambda^n)^{\frac{n-1}{n}},
\]
we deduce
\[
\frac{d}{dr}(f-\lambda^n)^{\frac{1}{n}} \geq \omega_n^{1/n} (1- O(r)),
\]
and after integration
\[
f\geq  \omega_n (r-C\lambda)^n (1- O(r) )
\]
as required.
\end{proof}

\begin{rmk}
In the context of special Lagrangians in $\C^n$, a standard way to deduce the monotonicity formula is to compare the volume of $L'\cap B(r)$ with the cone over $L'\cap \partial B(r)$. If we follow this strategy for smooth Lagrangians inside $\C^n$ with $|\theta|\leq \epsilon$, then we would deduce
\[
\text{Vol}(L'\cap B(r)) \leq \frac{1}{\cos \epsilon} \int_{L'\cap B(r)} \text{Re}\Omega \leq   \frac{r}{n\cos \epsilon} \mathcal{H}^{n-1}(L'\cap B(r)),
\]
whence $\frac{d}{dr} \log \text{Vol}(L'\cap B(r)) \geq \frac{n\cos \epsilon}{r}$. Unfortunately, this would only prove the monotonicity of $\text{Vol}(L'\cap B(r)) r^{-n\cos \epsilon}$, which is \emph{not sufficient} to deduce the volume lower bound above.

We think it is an interesting question whether $e^{Cr}\text{Vol}(L'\cap B(r)) r^{-n}$ is monotone under the mere hypothesis that $\norm{\theta}_{C^0}$ is small, with no assumption directly on the mean curvature.
\end{rmk}

\subsection{Hausdorff distance bound}

We still need to show that all the points of $L'$ must stay close to $L$. The strategy is to first show that such points occupy almost the full measure of $L'$, and then use a monotonicity inequality argument to extend this to all points.

We record a weighted volume upper bound:

\begin{lem}\label{volumeupperbound}
If $\text{Vol}(L'\cap \{ \theta_{L'}>\epsilon  \})\leq \epsilon^n$, then $\int_{L'}e^{-\rho}dvol_{L'} \leq \frac{1}{\cos \epsilon} \int_{L'}\text{Re}\Omega+C\epsilon^n$.
\end{lem}

\begin{proof}
On the subset $L'\cap \{ \theta_{L'}\leq\epsilon  \}$, we integrate $e^{-\rho}dvol_{L'}\leq \frac{1}{\cos \epsilon} \text{Re}\Omega|_{L'}$.
\end{proof}

\begin{lem}
In the setting of Theorem \ref{main}, 
\begin{equation}\label{almostfullmeasure}
\int_{L'\cap \{   \text{dist}(\cdot, L)\leq \epsilon^{1/4n} \} } e^{-\rho}dvol_{L'} \geq (1- C\epsilon^{1/{4n}}) \int_L \text{Re}\Omega.
\end{equation}

\end{lem}

\begin{proof}
We can take $\lambda=C\epsilon^{1/2n}$, so that by Cor. \ref{HalfHausdorff}, any $p\in L$ lies within distance $\lambda$ to $L'$. Our setting implies $\lambda\geq \epsilon^2$ and $\text{Vol}(\{\theta_{L'}>\epsilon  \})\leq \lambda^n$. By Prop. \ref{monotonicity}, and the fact that Euclidean balls and Riemannian balls only differ by relative error $O(r)$, we see that for $O(1)\geq r\geq C\lambda$, 
\[
\text{Vol}_g (L'\cap B_g(p,r)) \geq \omega_n (r-C\lambda)^n (1-Cr),\quad \forall p\in L.
\]
The constants are uniform for $p\in L$, so after integration
\[
\begin{split}
& \int_L e^{-\rho(p)}\text{Vol}_g (L'\cap B_g(p,r)) dvol_L(p) 
\\
\geq  & \omega_n (r-C\lambda)^n (1-Cr)\int_L e^{-\rho}dvol_L 
\\
\geq & \omega_n (r-C\lambda)^n (1-Cr)\int_L \text{Re}\Omega.
\end{split}
\]
The LHS can be computed by Fubini's theorem as
\[
\int_{L'} e^{-\rho(q)}dvol_{L'}(q) \int_{L\cap B_g(q,r)} e^{\rho(q)-\rho(p)} dvol_L(p).
\]
The function $\rho$ is smooth, so $|\rho(q)-\rho(p)|\leq Cr$. The Lagrangian $L$ has fixed $C^1$-regularity bound, so for small radius $r\ll 1$,
\[
\int_{L\cap B_g(q,r)} dvol_L(p)\leq (1+Cr) \omega_n r^n.
\]
Combining the above,
\[
\begin{split}
& \omega_n (r-C\lambda)^n (1-Cr)\int_L \text{Re}\Omega\leq \int_L e^{-\rho(p)}\text{Vol}_g (L'\cap B_g(p,r)) dvol_L(p) 
\\
\leq & (1+Cr) \omega_n r^n \int_{L'\cap \{   \text{dist}(\cdot, L)\leq r \} } e^{-\rho(q)}dvol_{L'}(q),
\end{split}
\]
whence for $C\lambda\leq r\leq O(1)$,
\begin{equation}\label{distanceweakL1}
\int_{L'\cap \{   \text{dist}(\cdot, L)\leq r \} } e^{-\rho(q)}dvol_{L'}(q) \geq (1- Cr- C\frac{\lambda}{r}) \int_L \text{Re}\Omega. 
\end{equation}
The choice $r= \epsilon^{1/4n}$ yields the claim.
\end{proof}

\begin{rmk}
In the above, we assumed $L$ is embedded.
If $L$ is merely immersed with transverse self intersections, then the inequality \[
\int_{L\cap B_g(q,r)} dvol_L(p)\leq (1+Cr) \omega_n r^n.
\]
will break down for $q$ in the $r$-neighbourhood of the self intersection points. Essentially the same proof would then give
\[
\int_{L'\cap \{   \text{dist}(\cdot, L)\leq \epsilon^{1/4n} \} } e^{-\rho}dvol_{L'} \geq (1- C\epsilon^{1/{4n}}) \int_L \text{Re}\Omega -Cr^n.
\]
Since $r^n\ll \epsilon^{1/4n}$, the $r^n$ term can be absorbed.

\end{rmk}

We can finally prove the remaining half of the \textbf{Hausdorff distance bound}:

\begin{prop}
In the setting of Theorem \ref{main}, 
\begin{equation}
\sup_{p\in L'} \text{dist}(p,L)\leq C\epsilon^{ \frac{1}{4n^2}  }.
\end{equation}
\end{prop}

\begin{proof}
The assumption that $[L]=[L']\in H_n(X,\Q)$, together with Lemma \ref{volumeupperbound} imply
\[
\int_{L'}e^{-\rho}dvol_{L'} \leq \frac{1}{\cos \epsilon} \int_{L}\text{Re}\Omega+C\epsilon^n\leq (1+C\epsilon^2) \int_{L}\text{Re}\Omega.
\]
Combined with Lemma \ref{almostfullmeasure},
\[
\int_{L'\cap \{   \text{dist}(\cdot, L)\geq \epsilon^{1/4n} \} } e^{-\rho}dvol_{L'} \leq C\epsilon^{1/{4n}} \int_L \text{Re}\Omega.
\]
Assume $p\in L'$ with $\text{dist}(p,L)=r_p>\epsilon^{1/4n}$, so that $L'\cap B(p,r_p- \epsilon^{1/4n} )$ is contained inside  $L'\cap \{   \text{dist}(\cdot, L)\geq \epsilon^{1/4n} \}$. From Prop. \ref{monotonicity},
\[
C(r_p-\epsilon^{1/4})^n \leq \text{Vol}( L'\cap  B(p,r_p- \epsilon^{1/4n} ) ) \leq C\epsilon^{1/{4n}} \int_L \text{Re}\Omega.
\]
Consequently $r_p\leq C\epsilon^{1/4n^2}$.
\end{proof}

\subsection{Flat norm distance bound}

Recall the flat norm distance between $L$ and $L'$ is defined as
\[
\inf \{  \text{Mass}(T)+ \text{Mass}(R)  : L-L'= \partial T+ R             \},
\]
where $T, R$ are integral currents of dim $n+1$ and $n$ respectively. Intuitively, the Hausdorff distance bound is an $L^\infty$ type bound on the distance function, while the flat norm behaves like an $L^1$ type bound. In the Corollary below we leverage the previous information to get a slightly better exponent than the obvious bound from $L^\infty\to L^1$, but we expect this is still not sharp.

\begin{cor}
In the setting of Theorem \ref{main}, there is an integral current $T$ with $\partial T=L'-L$, such that $\text{Mass}(T)\leq C\epsilon^{\frac{1}{4n^2}+ \frac{1}{4n} }$. 
\end{cor}

\begin{proof}
The Hausdorff bound shows that $L'$ lies within a small $C^0$-neighbourhood of $L$. Using the controlled $C^1$-regularity of $L$, we can view the $C^0$-neighbourhood as its normal bundle, and obtain a projection map $\pi: L'\to L$.  Since $L$ is homologous to $L'$ by assumption, $\pi$ has degree one. We form an $(n+1)$-dimensional integral current $T$ from the union of the line segments joining the points $q\in L'$ to $\pi(q)\in L$. Up to the choice of orientation sign, $\partial T= L'-\pi_*(L')= L'-L$. The lengths of the segments are comparable to $\text{dist}(q,L)$. Thus
\[
\text{Mass}(T)\leq C \int_{L'} \text{dist}(\cdot, L) dvol_{L'} = C \int_0^{\sup \text{dist}} dr \int_{ L'\cap  \{\text{dist}(\cdot, L)\geq r \} }  dvol_{L'}.
\]
Recall the weak $L^1$-type estimate (\ref{distanceweakL1}), which implies
\[
\int_{ L'\cap  \{\text{dist}(\cdot, L)\geq r \} }  dvol_{L'} \leq C\begin{cases}
1,\quad r=O(\epsilon^{1/2n} ),
\\
\frac{ \epsilon^{1/2n}} {r}, \quad \epsilon^{1/2n}\leq r\leq \epsilon^{1/4n},
\\
\epsilon^{1/4n}, \quad  r>\epsilon^{1/4n}. 
\end{cases}
\]
By the Hausdorff distance estimate, $\sup \text{dist}(\cdot, L)\leq C\epsilon^{1/4n^2}$. 
 Combining the above gives the mass bound.
\end{proof}

\begin{rmk}
If $L$ is immersed rather than embedded, then $\pi$ would be only Lipschitz near the self intersection points, but the above argument is unaffected.
\end{rmk}

\subsection{F-metric bound}

We view the $C^0$ neighbourhood of $L$ as its normal bundle, which has a projection $\pi$ to its zero section $L$. Normal geodesic flow allows us to identify the tangent spaces $T_qX$ with $T_{\pi(q)}X$. We now show that $L'$ is graphical over $L$ with small norm, away from a set with small measure. Denote $\vec{T}(q)$ as the unit $n$-vector $e_1\wedge \ldots e_n$ on $L'$ where $e_i$ is an oriented orthonormal frame of $T_qL'$, and $\vec{T}(\pi(q))$ as the analogous unit $n$-vector from the oriented orthonormal basis of $T_{\pi(q)}L$.

\begin{prop}\label{almostgraphical}
In the setting of Theorem \ref{main},  away from a subset $E\subset L$ with $\text{Vol}(E)+ \text{Vol}(\pi^{-1}(E))\leq C\epsilon^{1/4n}$, the projection $\pi:L'\setminus \pi^{-1}(E)\to L\setminus E$ is bijective, and 
\[
\int_{ L'\setminus \pi^{-1}(E) } |\vec{T}(q)- \vec{T}(\pi(q))|^2 dvol(q) \leq C\epsilon^{1/4n}.
\]
\end{prop}

\begin{proof}
The projection $\pi: L'\to L$ is almost metric decreasing: 
\[
|d\pi|_{T_q L'} |\leq 1+ O(\text{dist}(q, \pi(q))).
\]
As a general rule, the error coming from comparing ambient K\"ahler structures at $q$ and $\pi(q)$ is bounded by $O(\text{dist}(q,\pi(q)))$. Let $E_1'= \{ \text{dist}(\cdot, L)\geq \epsilon^{1/4n}   \}\subset L'$, then $E_1'$ has measure bounded by $C\epsilon^{1/4n}$ by Lemma \ref{volumeupperbound} and (\ref{almostfullmeasure}).

Since $\pi:L'\to L$ has degree one, it is surjective.
Let $E$ be the subset of $L $ with at least two preimages. The volume almost decreasing property of $d\pi$, together with the assumption $|\theta|\leq \epsilon$ on $L$, imply
\[
2\int_{E}\text{Re}\Omega=
2\int_{E}  e^{-\rho}dvol (1+ O(\epsilon)) \leq  (1+ O(\epsilon^{1/4n^2})) \int_{ \pi^{-1} (E)} e^{-\rho}dvol.
\]
In particular
\begin{equation}\label{doublecovervol}
\text{Vol}(E)+ \text{Vol}(\pi^{-1}(E))\leq 
C\int_{ \pi^{-1} (E)} e^{-\rho}dvol\leq C' ( \int_{ \pi^{-1} (E)} e^{-\rho}dvol- \int_E  \text{Re}\Omega).
\end{equation}

On the other hand, at the points  $q\in L'\setminus \pi^{-1}(E)\cup E_1'$, the difference between the tangent planes is bounded by
\[
|\vec{T}(q)- \vec{T}(\pi(q))|^2 \leq C (1- |d\pi(\vec{T})| + O(\epsilon^{1/4n} )).
\]
Thus
\[
\begin{split}
& \int_{ L'\setminus \pi^{-1}(E)} |\vec{T}(q)- \vec{T}(\pi(q))|^2 dvol(q)  
\\
 \leq & C\epsilon^{1/4n} + C\int_{L'\setminus \pi^{-1}(E)}  (1- |d\pi(\vec{T})|) dvol
\\
\leq & C\epsilon^{1/4n} + C'\int_{L'\setminus \pi^{-1}(E)}  (1- |d\pi(\vec{T})|) e^{-\rho}dvol.
\end{split}
\]
Here the contribution from $E_1'$ is absorbed into $C\epsilon^{1/4n}$. 
From the pointwise inequality on $L'\setminus \pi^{-1}(E)\cup E_1'$,
\[
|d\pi(\vec{T})|e^{-\rho}=(1- O(\epsilon^{1/4n})) |d\pi(\vec{T})| |\text{Re}\Omega|_L|\geq  \text{Re}\Omega(\pi_*\vec{T})  (1- O(\epsilon^{1/4n})) ,
\]
%and since $\text{Re}\Omega|_{L'}$ is $O(\epsilon)$ close to $e^{-\rho}dvol$ away from negligible measure, 
the above is bounded by
\[
C\epsilon^{1/4n} + C'(\int_{L'\setminus \pi^{-1}(E)} e^{-\rho}dvol  - \int_{L\setminus E} \text{Re} \Omega).
% \leq C\epsilon^{1/4n^2} + C(\int_{L'\setminus \pi^{-1}(E)} \text{Re}\Omega  - \int_{L\setminus E} \text{Re} \Omega).
\]
The coefficient $C'$ can be arranged as the same (large) constant appearing in (\ref{doublecovervol}). Adding the two contributions, and envoking the weighted volume upper bound in Lemma \ref{volumeupperbound},
\[
\begin{split}
& \text{Vol}(E)+ \text{Vol}(\pi^{-1}(E)) +\int_{ L'\setminus \pi^{-1}(E)} |\vec{T}(q)- \vec{T}(\pi(q))|^2 dvol(q) 
\\
\leq & C\epsilon^{1/4n} + C'(\int_{L'} e^{-\rho}dvol  - \int_{L} \text{Re} \Omega)
\\
\leq & C\epsilon^{1/4n} + C'(\int_{L'} \text{Re}\Omega  - \int_{L} \text{Re} \Omega)= C\epsilon^{1/4n},
\end{split}
\]
as required.
\end{proof}

\begin{rmk}
	If $L$ is immersed instead of embedded, then we include the $O(\epsilon^{1/4n^2})$ neighbourhood of the self intersections into $E$, and the rest of the arguments run almost verbatim.
\end{rmk}

Recall the distance between two varifolds $L,L'$ is measured by \textbf{F-metric}
\[
\sup \{  |\int_L f(p, T_p L) dvol_L  -\int_{L'} f(q, T_q L' ) dvol_{L'}  |   \}.
\]
where $f$ ranges over all functions on the Grassmannian bundle of $n$-dimensional tangent subspaces over $X$, with $|f|\leq 1$ and Lipschitz constant at most one. Given a uniform upper bound on the mass, the $F$-metric induces the same topology as the weak topology on varifolds (because continuous functions can be approximated by Lipschitz functions).

\begin{cor}
In the setting above, the $F$-metric between $L, L'$ is bounded by $C\epsilon^{1/8n}$. 
\end{cor}

\begin{proof}
Since $|f|\leq 1$, the integrals on $E, E_1'$ and $\pi^{-1}(E)$ are bounded by $C\epsilon^{1/4n}$. At $q\in L'\setminus \pi^{-1}(E)\cup E_1'$, by the Lipschitz bound on $f$,
\[
|f(q, T_q L')- f(\pi(q), T_{\pi(q)}L)| \leq |q-\pi(q)|+ |\vec{T}(q)- \vec{T}(\pi(q))| \leq C\epsilon^{1/4n}+ |\vec{T}(q)- \vec{T}(\pi(q))|.
\]
The discrepancy between the volume forms $\pi^*dvol_L$ and $dvol_{L'}$ is bounded by $C |\vec{T}(q)- \vec{T}(\pi(q))|^2 dvol_{L'}$. Consquently, the F-metric is bounded by
\[
C\epsilon^{1/4n}+ \int_{L'\setminus \pi^{-1}(E)} |\vec{T}(q)- \vec{T}(\pi(q))| dvol_{L'} + C\int_{L'\setminus \pi^{-1}(E)} |\vec{T}(q)- \vec{T}(\pi(q))|^2 dvol_{L'}.
\]
Using Prop. \ref{almostgraphical} and Cauchy-Schwarz, this is bounded by $C\epsilon^{1/8n}$.
\end{proof}

\begin{Acknowledgement}
The author is currently a Clay Research Fellow based at MIT. He thanks P. Seidel for discussions.	
\end{Acknowledgement}


\begin{thebibliography}{7}
	
	
	
	
\bibitem{AbouzaidImagi} 
Mohammed Abouzaid, Yohsuke Imagi. Nearby Special Lagrangians. 	arXiv:2112.10385.




\bibitem{AkahoJoyce} 
Akaho, Manabu; Joyce, Dominic. Immersed Lagrangian Floer theory. J. Differential Geom. 86 (2010), no. 3, 381--500. 



\bibitem{Almgrenisoperimetric} 
Almgren, F. Optimal isoperimetric inequalities. Indiana Univ. Math. J. 35 (1986), no. 3, 451--547.



\bibitem{Arnold}
Arnol'd, V. I. Dynamics of complexity of intersections. Bol. Soc. Brasil. Mat. (N.S.) 21 (1990), no. 1, 1--10.



\bibitem{FOOO}
Fukaya, Kenji; Oh, Yong-Geun; Ohta, Hiroshi; Ono, Kaoru. Lagrangian intersection Floer theory: anomaly and obstruction. Part I. AMS/IP Studies in Advanced Mathematics, 46.1. American Mathematical Society, Providence, RI; International Press, Somerville, MA, 2009. xii+396 pp.

   

\bibitem{Li} 
Li, Yang. Thomas-Yau conjecture and holomorphic curves. 	arXiv:2203.01467.


\bibitem{JoyceImagi}
Imagi, Yohsuke; Joyce, Dominic; Oliveira dos Santos, Joana. Uniqueness results for special Lagrangians and Lagrangian mean curvature flow expanders in $\mathbb{C}^m$. Duke Math. J. 165 (2016), no. 5, 847--933.




\bibitem{Imagi}
Imagi, Yohsuke. A uniqueness theorem for gluing calibrated submanifolds. Comm. Anal. Geom. 23 (2015), no. 4, 691--715.


\bibitem{Morgan}
Morgan, Frank. Geometric measure theory. A beginner's guide. Fourth edition. Elsevier/Academic Press, Amsterdam, 2009. viii+249 pp. 



\bibitem{Neves}
Neves, André. Recent progress on singularities of Lagrangian mean curvature flow. Surveys in geometric analysis and relativity, 413--438, Adv. Lect. Math. (ALM), 20, Int. Press, Somerville, MA, 2011. 



\bibitem{Seidel}
Seidel, Paul. Categorical dynamics. available on Seidel's webpage.



\bibitem{Thomas}
Thomas, R. P. Moment maps, monodromy and mirror manifolds. Symplectic geometry and mirror symmetry (Seoul, 2000), 467--498, World Sci. Publ., River Edge, NJ, 2001.


\bibitem{ThomasYau} 
Thomas, R. P.; Yau, S.-T. Special Lagrangians, stable bundles and mean curvature flow. Comm. Anal. Geom. 10 (2002), no. 5, 1075--1113.
	
\end{thebibliography}
\end{document}